\documentclass{article}

\author{Fabio Gobbi \hspace{0.4cm} Sabrina Mulinacci\\
\vspace{0.4cm} University of Bologna - Department of Statistics}

\usepackage{amsmath}
\usepackage{amsthm}
\usepackage{amsfonts}
\usepackage{amssymb}
\usepackage{latexsym}
\usepackage[dvips]{graphicx}
\usepackage{pstricks}
\usepackage{epsfig}
\usepackage{a4}
 \oddsidemargin=\evensidemargin
\newtheorem{theorem}{Theorem}[section]
\newtheorem{corollary}[theorem]{Corollary}

\newtheorem{proposition}{Proposition}[section]

\begin{document}

\title{$\beta$-mixing and moments properties of a non-stationary copula-based Markov process}

\maketitle

\begin{abstract}
This paper provides conditions under which a non-stationary copula-based Markov process is $\beta$-mixing. We introduce, as a particular case, a
convolution-based gaussian Markov process
which generalizes the standard random walk allowing the increments to be dependent.
 \end{abstract}
\bigskip


  \vspace{0.2cm}


\medskip

\section{Introduction}
In this paper we analyze the temporal dependence properties
satisfied by a discrete times non-stationary Markov process.
Temporal dependence is relevant since it permits to verify of how
well theoretical models explain temporal persistency observed in
financial data. Moreover, it is also a useful tool to establish
large sample properties of estimators for dynamic models. In
particular, in this paper we analyze the $\beta$-mixing property and we give sufficient conditions that ensure this property be satisfied.

In the
copula approach to univariate time series modelling, the finite
dimensional distributions are generate by copulas. Darsow et al.
(1992) provide necessary and sufficient conditions for a
copula-based time series to be a Markov process. Recent literature on this topic has mainly focused on the stationary case. Chen and Fan
(2006) introduce a copula-based strictly stationary first order
Markov process generated from $(G_0(\cdot),
C(\cdot,\cdot,\alpha_0))$ where $G_0(\cdot)$ is the invariant
distribution of $Y_t$ and $C(\cdot,\cdot,\alpha_0)$ is the
parametric copula for $(Y_{t-1},Y_t)$. The authors show that the
$\beta$-mixing temporal dependence measure is purely determined by
the properties of copulas and present sufficient conditions to
ensure that the process $(Y_t)_t$ based on gaussian and EFGM
copulas are geometric $\beta$-mixing. Beare (2010) shows that all
Markov models generated via symmetric copulas with positive and
square integrable densities are geometric $\beta$-mixing. Many
commonly used bivariate copulas without tail dependence such as
gaussian, EFGM and Frank copulas satisfy this condition. Chen et
al. (2009) show that Clayton, Gumbel and Student's $t$ copula
based Markov models are geometrically ergodic which is a stronger
condition than the geometric $\beta$-mixing.

In this paper we focus on Markov processes where some
 dependence between each state variable and increment is allowed and modeled through a time-invariant copula.
In particular, we introduce a gaussian Markov process, which is non-stationary and generalizes the classical gaussian random walk, and we
study related moments properties and provide conditions under wich the process is $\beta$-mixing.
\bigskip

The paper is organized as follows. Section \ref{markov} presents a general result on the
 $\beta$-mixing properties satisfied by non-stationary Markov processes.
Section \ref{gaussian} restricts the study to the gaussian case. Section \ref{conclusion} concludes.
\section{Copula-based Markov processes and $\beta$-mixing properties}\label{markov}
Throughout the paper ${\bf Y}=(Y_t)_{t\in\mathbb Z}$ is a discrete
time Markov process. Thanks to the seminal paper of  Darsow et al.
(1992), the markovianity of a stochastic process can be
characterized through a specific requirement that the copulas,
representing the dependence structure of the finite dimensional
distributions induced by the stochastic process (for a detailed
discussion on copulas see Nelsen (2006), Joe (1997), Cherubini et
al. (2012) and Durante and Sempi (2015)), must satisfy. In
particular, in Darsow et al (1992) it is proved that the
Chapman-Kolmogorov equations for transition probabilities are
equivalent to the requirement that, if $C_{i,j}$ is the copula
associated to the vector $(Y_i,Y_j)$, then
$$C_{s,t}(u,v)=C_{s,r}\ast C_{r,t}(u,v)=\int_0^1\frac{\partial}{\partial w}C_{s,r}(u,w)\frac{\partial}{\partial w}C_{r,t}(w,v)\, dw,\, \forall s<r<t.$$
As a consequence, since ${\bf Y}$ is a discrete times Markov
process, if we assume that the set of bivariate copulas
$C_{t,t+1}$ (representing the dependence structure of the
stochastic process at two adjacent times) is given for
$t\in\mathbb Z$ and $k>0$, then necessarily (we remind that the
$\ast$-operator is associative)
\begin{equation}\label{star1}C_{t,t+k}(u,v)=C_{t,t+k-1}\ast C_{t+k-1,t+k}(u,v)=C_{t,t+1}\ast C_{t+1,t+2}\ast\cdots\ast C_{t+k-1,t+k}(u,v).\end{equation}
\bigskip

Notice that, in the stationary case considered in Beare (2010), $C_{t,t+1}=C$ for all $t\in\mathbb Z$,
therefore all bivariate copulas $C_{t,t+k}$
are functions of the copula $C$ and of the lag $k$ and not of the
 time $t$. In this paper we extend the study to the more general non-stationary case. In particular
we analyze the temporal dependence problem with a special attention to mixing properties.

The notion of $\beta$-mixing was introduced by Volkonskii and
Rozanov (1959 and 1961) and was attribute there to Kolmogorov.
Given a (not necessarily stationary) sequence of random variables
${\bf Y}=(Y_t)_{t \in \mathbb{Z}}$, let $\mathcal{F}_{t}^{l}$ be
the $\sigma$-field $\mathcal{F}_{t}^{l}=\sigma(Y_t,t\leq t \leq
l)$ with $-\infty \leq t \leq l \leq +\infty$ and let
\begin{equation}\label{beta1}\tilde{\beta}(\mathcal{F}_{-\infty}^{t},\mathcal{F}_{t+k}^{+\infty})=\sup_{\{A_i\},\{B_j\}}\frac{1}{2}\sum_{i=1}^I \sum_{j=1}^J|\mathbb{
P}(A_i\cap B_j)-\mathbb{P}(A_i)\mathbb{P}(B_j)|,\end{equation}
where the
second supremum is taken over all finite partitions
$\{A_1,...A_I\}$ and $\{B_1,...B_J\}$ of $\Omega$ such that $A_i
\in \mathcal{F}_{-\infty}^t$ for each $i$ and $B_j \in
\mathcal{F}_{t+k}^{\infty}$ for each $j$. Define the following
dependence coefficient
$$\beta_k=\sup_{t \in \mathbb{Z}}\tilde{\beta}(\mathcal{F}_{-\infty}^{t},\mathcal{F}_{t+k}^{+\infty}).$$
We say that the sequence $(Y_t)_{t \in \mathbb{Z}}$ is
$\beta-$mixing (or absolutely regular) if $\beta_k \rightarrow 0$
as $k \rightarrow +\infty$.

In next Theorem we give conditions on the set of copulas $C_{t,t+1}$, $t\in\mathbb Z$ in order to
guarantee that the Markov resulting process is $\beta$-mixing. These conditions are based on specific requirements
on the maximal correlation coefficients of the copulas $C_{t,t+1}$. We remind that the maximal correlation $\eta $ of a copula $C$ is given by
$$\underset{f,g}\sup\left \vert \int_0^1\int_0^1f(x)g(x)C(dx,dy)\right \vert $$
where $f,g\in L^2([0,1])$, $\int_0^1f(x)dx=\int_0^1g(x)dx=0$ and $\int_0^1f^2(x)dx=\int_0^1g^2(x)dx=1$ and we refer to Beare (2010) and R\'enyi (1959) for more details.

\begin{theorem}\label{teo} Let ${\bf Y}=(Y_t)_{t \in \mathbb{Z}}$ be a
Markov process. Let $C_{t,t+1}$ be the copula associated to the vector
 $(Y_t,Y_{t+1})$ for $t\in\mathbb Z$ that we assume to be
absolutely continuous, with symmetric and square-integrable density
$c_{t,t+1}$ so that $(c_{t,t+1})_{t\in\mathbb Z}$ is uniformly bounded in $L^2([0,1])$.

If the maximal correlation coefficients $\eta_{t}$ of $C_{t,t+1}$ satisfy
\begin{equation}\label{ipo1}\hat\eta=\underset{t\in\mathbb Z}\sup\,\eta_t<1,\end{equation}
then ${\bf Y}$ is $\beta$-mixing.
\end{theorem}
\begin{proof} The proof follows that of Theorem 3.1 in Beare (2010) who
proves a similar result for stationary copula-based Markov
processes. First of all, since the stochastic process is Markovian, (\ref{beta1}) can be rewritten in terms of the cumulative distribution functions
of $(Y_t,Y_{t+k})$, $Y_{t}$ and $Y_{t+k}$ ($F_{t,t+k}$, $F_{t}$ and $F_{t+k}$, respectively) and the total
 variation norm $\|\cdot\|_{TV}$ (see Bradley, 2007)
and then, applying Sklar's theorem, we can write
$$\begin{aligned}\tilde{\beta}(\mathcal{F}_{-\infty}^{t},\mathcal{F}_{t,t+k}^{+\infty})&= \frac{1}{2}\parallel
F_{t,t+k}(x,y)-F_t(x)F_{t+k}(y)\parallel_{TV}=\\
&=\frac{1}{2}\parallel
C_{t,t+k}(F_t(x),F_{t+k}(y))-F_t(x)F_{t+k}(y)\parallel_{TV}\leq\\
&\leq \frac{1}{2}\parallel
C_{t,t+k}(u,v)-uv)\parallel_{TV}.\end{aligned}$$ From
(\ref{star1}) it follows that all bivariate copulas of type
$C_{t,t+k}$ for $t\in\mathbb Z$ and $k \geq 1$ are absolutely
continuous: let us denote their density as $c_{t,t+k}$. Then
$$\tilde{\beta}(\mathcal{F}_{-\infty}^{t},\mathcal{F}_{t+k}^{+\infty})\leq \frac{1}{2}\parallel
c_{t,t+k}(u,v)-1\parallel_{\mathbb{L}^1}\leq \frac{1}{2} \parallel
c_{t,t+k}(u,v)-1\parallel_{\mathbb{L}^2}$$
and
$$\beta_k\leq \frac{1}{2} \sup_{t \in \mathbb{Z}}\parallel
c_{t,t+k}(u,v)-1\parallel_{\mathbb{L}^2}.$$
Since $c_{t,t+1}$ is a
symmetric square-integrable joint density with uniform margins, it admits the following series expansion in terms of
a complete orthonormal sequence $(\phi_i)_{i\geq 1}$
in $\mathbb{L}^2[0,1]$,
$$c_{t,t+1}(u,v)=1+\sum_{i=1}^{+\infty} \lambda_{i,t} \phi_i(u)\phi_i(v),$$
where the eigenvalues $(\lambda_{i,t})_i$ form a
square-summable sequence of nonnegative real numbers: notice that, as proved in Lancaster(1958)
\begin{equation}\label{maggio1}\underset{i\geq 1}\max \,\lambda _{i,t}=\eta_t.\end{equation}

Applying (\ref{star1}), we get
$$c_{t,t+k}(u,v)=1+\sum_{i=1}^{+\infty} \left(\prod_{j=0}^{k-1}\lambda_{i,t+j}\right) \phi_i(u)\phi_i(v).$$
Then, using (\ref{maggio1}) and (\ref{ipo1}), we get
\begin{equation}\begin{aligned}\parallel c_{t,t+k}(u,v)-1\parallel_{\mathbb{L}^2}&=\left\| \sum_{i=1}^{+\infty}
 \left(\prod_{j=0}^{k-1}\lambda_{i,t+j}\right) \phi_i(u)\phi_i(v) \right\|_{\mathbb{L}^2}= \left[\sum_{i=1}^{+\infty} \left(\prod_{j=0}^{k-1}\lambda^2_{i,t+j}\right)\right]^{1/2}=\\
 &= \left[\sum_{i=1}^{+\infty} \lambda^2_{i,t}\left(\prod_{j=1}^{k-1}\lambda^2_{i,t+j}\right)\right]^{1/2}
 \leq
 \left[\sum_{i=1}^{+\infty} \lambda^2_{i,t}
 \left(\prod_{j=1}^{k-1}\eta^2_{t+j}\right)\right]^{1/2}\leq\\
&\leq \hat\eta^{k-1}\left[\sum_{i=1}^{+\infty}
\lambda^2_{i,t} \right]^{1/2}=\hat\eta^{k-1}\parallel
 c_{t,t+1}(u,v)-1\parallel_{\mathbb{L}^2}.\end{aligned}\end{equation}
Therefore
$$\beta_k\leq \frac{1}{2}\hat\eta^{k-1} \sup_{t\in \mathbb{Z}}\parallel
c_{t,t+1}(u,v)-1\parallel_{\mathbb{L}^2}$$
which, since $(c_{t,t+1})_{t\in\mathbb Z}$ is uniformly bounded in $L^2([0,1])$, tends to zero as
$k\rightarrow +\infty$.
\end{proof}
\section{A gaussian convolution-based Markov process}\label{gaussian}
From now on, we assume that the Markov process ${\bf Y}$ is obtained through
\begin{equation}\label{C_UR1}Y_t= Y_{t-1} +
\xi_t,\qquad Y_0=0,\end{equation} where $(\xi_t)_{t\geq 1}$ is a
sequence of identically distributed random variables such that
$\xi_t$ is dependent on $Y_{t-1}$ for each $t$. The dependence
structure is modelled by a time-invariant copula function $C$.
\smallskip

The process defined in (\ref{C_UR1}) is not stationary. However,
we can determine the distribution of $Y_t$ for each $t$ thanks to
the $C$-convolution operator (denoted by $\stackrel{C}{\ast}$),
introduced in Cherubini et al. (2011) as a tool to recover the
distribution of the sum of two dependent random variables. As
shown in Cherubini et al. (2011, 2012 and 2015) the
$C$-convolution technique may be used in the construction of
dependent increments stochastic processes like (\ref{C_UR1}). More
precisely, if $F_{t-1}$ is the cumulative distribution function of
$Y_{t-1}$ and $H_{t}$ that of $\xi_t$, we may recover the
cumulative distribution function of $Y_{t}$ iterating the
$C$-convolution for all $t$
\begin{equation}\label{CURDist}F_t(y_t)=(F_{t-1}\stackrel{C}{\ast}H_{t})(y_t)=\int_0^{1}D_1C(w,H_{t}(y_t-
F^{-1}_{t-1}(w)))dw, \quad t\geq 2\end{equation}
while, the copula associated to $(Y_{t-1},Y_t)$ is
\begin{equation}\label{cc1}C_{t-1,t}(u,v)=\int_0^{u}D_1C(w,H_{t}(F_t^{-1}(v)-
F^{-1}_{t-1}(w)))dw, \quad t\geq 2\end{equation}
where $D_1C(u,v)=\frac{\partial}{\partial u}C(u,v)$.

Equations (\ref{CURDist}) and (\ref{cc1}) provide the ingredients
to construct discrete times Markov processes according to Darsow
et al. (1992).

Our model (\ref{C_UR1}) is a sort of a
modified version of a random walk process where the independence
assumption for the innovations $(\xi_t)_{t\geq 1}$ is no longer
required: however, its weakness is that in most cases the
distribution function cannot be expressed in closed form and it
may be evaluated only numerically.
\medskip

From now on we assume that
innovations $(\xi_t)_{t\geq 1}$ are gaussian identically distributed with
zero mean and standard deviation $\sigma_{\xi}$ and that the copula
between $Y_{t-1}$ and $\xi_t$ is a (stationary) gaussian copula
with constant parameter $\rho$ for all $t$. This way, the distribution of
$Y_t$ is gaussian for all $t$ and, more specifically, in section 4.3.1 of Cherubini et al. (2016) it is shown that

\begin{equation}\label{Gaussian_C_UR1}Y_t\sim N(0,V_t^{2}),\end{equation} where
\begin{equation}\label{VarGaussian_C_UR1}V_t^{2}=Var(Y_t)=V_1^{2}+(t-1)\sigma_{\xi}^2+
2\rho \sigma_{\xi} \sum_{i=1}^{t-1}V_{t-i}, \quad t\geq
2,\end{equation} where $V_1^2=\sigma_{\xi}^2$ since by assumption
$Y_1=\xi_1$. Moreover, the copula between $Y_{t}$ and $Y_{t+1}$ is
gaussian with parameters
$$\tau_{t,t+1}=\frac{ V_{t}+\rho \sigma_{\xi}}{V_{t+1}}, \quad t\geq 2$$
since $\mathbb{E}[Y_t Y_{t+1}]=V_t^2+\rho V_t \sigma_{\xi}$.

The limiting behavior of the standard deviation $V_t$ has also been
analyzed in Cherubini et al. (2016) where it is proved that
$$
\underset{t\rightarrow +\infty}\lim V_t =\left\{%
\begin{array}{ll}
    -\frac{\sigma_{\varepsilon}}{2 \rho}, & \hbox{ if \ $\rho \in (-1,0)$} \\
    +\infty, & \hbox{otherwise.} \\
\end{array}%
\right.$$
Notice that only in case of negative correlation with the increments, the standard deviation of the levels does not explode:
in the following we will restrict the analysis to the case $\rho\in (-1,0)$.

\subsection{Moments and autocorrelation function}
In this subsection we study the behavior of moments and
autocorrelation functions of the process $(Y_t)_{t\geq 1}$ when $t
\rightarrow +\infty$. It is just the case to recall that in the
standard random walk model the $k$-th order autocorrelation
function of $(Y_t)_{t\geq 1}$ tends to $1$ as $t\rightarrow
+\infty$, for each lag $k$. In our more general setting, this is
no longer true. The limit of the $k$-th order autocorrelation
function of $(Y_t)_{t\geq 1}$ is a function of $k$ and $\rho$ as
the following proposition shows.
\begin{proposition} Let $\rho\in(-1,0)$. The $k$-th order autocorrelation function of $(Y_t)_{t\geq
1}$ tends to $(1-2 \rho^2)^{k}$ for any $k\geq 1$ as $t
\rightarrow +\infty$.
\end{proposition}
\begin{proof}
As proved in section 4.3.1 in Cherubini et al. (2016), using the fact that the $\ast$-product of two gaussian copulas has a parameter given by the product of the
parameters of the copulas involved in the $\ast$-product, we have that the
copula between $Y_{t}$ and $Y_{t+k}$ is gaussian with parameter
$$\tau_{t,t+k}=\prod_{s=0}^{k-1} \frac{V_{t+s}+\rho \sigma_{\xi}}{V_{t+s+1}}.$$
Therefore, since as $t \rightarrow +\infty$ and for any $s\geq 1$
\begin{equation}\label{f1}\frac{V_{t+s}+\rho \sigma_{\xi}}{V_{t+s+1}}\rightarrow \frac{-\frac{\sigma_{\xi}}{2 \rho}+\rho \sigma_{\xi}^2}{-\frac{\sigma_{\xi}}{2 \rho}}=
1-2 \rho^2,\end{equation}
we easily get the result.
\end{proof}
On the other hand, the innovations $(\xi_t)_{t\geq 1}$ are no longer serially independent as in the random walk case and the $k$-th order autocorrelation function approaches to a limit
which again depends on $\rho$ and $k$.

\begin{proposition} Let $\rho\in(-1,0)$. The $k$-th order autocorrelation function of $(\xi_t)_{t\geq
1}$ tends to $-\rho^2 (1-2 \rho^2)^{k-1}$ for any $k\geq 1$ as $t
\rightarrow +\infty$.
\end{proposition}
\begin{proof}
We compute first the autocovariance of order $k$, with $k\geq 1$,
$\mathbb{E}[\xi_t \xi_{t+k}]$. We have
$$\begin{aligned}\mathbb{E}[\xi_t
\xi_{t+k}]&=\mathbb{E}[(Y_t-Y_{t-1})(Y_{t+k}-Y_{t+k-1})]=\\
&=\mathbb{E}[Y_t Y_{t+k}]-\mathbb{E}[Y_t
Y_{t+k-1}]-\mathbb{E}[Y_{t-1} Y_{t+k}]+\mathbb{E}[Y_{t-1}
Y_{t+k-1}]=\\
&=\tau_{t,t+k}V_t V_{t+k}-\tau_{t,t+k-1}V_t
V_{t+k-1}-\tau_{t-1,t+k}V_{t-1}
V_{t+k}+\tau_{t-1,t+k-1}V_{t-1} V_{t+k-1}.\end{aligned}$$
Since for any fixed
$k\geq 1$, $\tau_{t,t+k} \rightarrow (1-2 \rho^2)^k$ and $V_t\rightarrow -\frac{\sigma_{\xi}}{2\rho}$ as $t
\rightarrow +\infty$ we get
$$\begin{aligned}\mathbb{E}[\xi_t\xi_{t+k}]\rightarrow &\frac{\sigma^2_{\xi}}{4 \rho^2}\left[(1-2 \rho^2)^k-(1-2 \rho^2)^{k-1}
-(1-2 \rho^2)^{k+1}+(1-2 \rho^2)^k\right]=\\
&=-\rho^2 \sigma_{\xi}^2 (1-2 \rho^2)^{k-1}, \quad as \ t \rightarrow +\infty.\end{aligned}$$
Moreover, it is immediate to find the statement of the proposition
since as $t \rightarrow +\infty$
$$corr(\xi_t \xi_{t+k})\rightarrow -\rho^2 (1-2 \rho^2)^{k-1}.$$
\end{proof}

\subsection{$\beta$-mixing properties}
In our gaussian framework $c_{t,t+1}$ is the
density of a gaussian copula for which it is well known that the
maximal correlation coefficient is equal to the absolute value of
the simple correlation coefficient (see Lancaster, 1957). Therefore, according to the notation of Theorem \ref{teo},
for each $t$,
$\eta_t=|\tau_{t,t+1}|$. The following results, which is an application of Theorem \ref{teo}, holds
\begin{corollary}
The Markov process defined by (\ref{C_UR1}) with  $\xi_t\sim N(0,\sigma_{\xi})$ and $\rho\in (-1,0)$ is $\beta$-mixing.
\end{corollary}
\begin{proof}
Firstly notice that
$$|\tau_{t,t+1}|<1, \, \forall t.$$
In fact this is equivalent to $(V_t+\rho\sigma_\xi)^2< V_{t+1}^2$ which is always verified since $\rho^2<1$ by assumption.

Thanks to (\ref{f1}), since $\vert 1-2\rho^2\vert <1$, we have that
$|\tau_{t,t+1}|=\eta_t$ is bounded by a constant smaller than 1, that is (\ref{ipo1}) is satisfied.
Furthermore, it is not hard to prove that for any $t$
$$\parallel c_{t,t+1}(u,v)-1\parallel_{\mathbb{L}^2}=\frac{\tau_{t,t+1}^2}{1-\tau^2_{t,t+1}}\leq \frac{\hat \eta^2}{1-\hat\eta^2}.$$
Thus Theorem \ref{teo} applies.

\end{proof}

\section{Concluding remarks}\label{conclusion}
In this paper we provide conditions under which a non-stationary copula-based Markov process is $\beta$-mixing. Our results represent a generalization of those in
Beare (2010), where the author considers the stationary case. Our analysis is focused on the particular case of a gaussian Markov process with dependent
increments that represents a generalization of the standard gaussian random walk. In this particular non-stationary setting it is proved that the $k$-th order autocorrelation function
of the process does not converge to $1$, as in the random walk case, but to a quantity that depends on the lag and the correlation between the state
variable and the innovation, which is assumed to be time-invariant.
Additionally, it is proved that the process satisfies the conditions required to be $\beta$-mixing.


\begin{thebibliography}{9}

\bibitem{} Beare B. (2010): "Copulas and temporal dependence",
\emph{Econometrica}, 78(1), 395-410.

\bibitem{} Bradley R.C. (2007): \emph{Introduction to strong mixing conditions},
vols. 1-3. Kendrick Press. herber City.


\bibitem{} Chen X., Fan Y. (2006): "Estimation of Copula-Based
Semiparametric Time Series Models", \emph{Journal of
Econometrics}, 130, 307–335.

\bibitem{} Chen X., Wu W.B., Yi Y. (2009): "Efficient estimation of copula-based semiparametric Markov models", \emph{
The Annals of Statistics}, 37(6B) 4214–4253.

\bibitem{} Cherubini U., Gobbi F., Mulinacci S., (2016)
\emph{"Convolution Copula Econometrics"}, SpringerBriefs in
Statistics.

\bibitem{} Cherubini U., Gobbi F., Mulinacci S., Romagnoli S. (2012): \emph{Dynamic Copula Methods in Finance}, John Wiley \&
Sons.

\bibitem{} Cherubini, U., Mulinacci S., Romagnoli S. (2011) "A Copula-based Model of Speculative Price Dynamics in Discrete
Time", Journal of Multivariate Analysis, 102, 1047-1063.

\bibitem{DNO} Darsow W.F. - Nguyen B. - Olsen E.T.(1992): "Copulas and Markov Processes", \emph{Illinois Journal of Mathematics}, 36,
600-642.

\bibitem{DS} Durante F., Sempi C. (2015) \emph{Principles of copula theory}, Boca Raton: Chapman and Hall/CRC

\bibitem{LPP} Joe H. (1997): \emph{Multivariate Models and Dependence Concepts},
Chapman \& Hall, London
\bibitem{} Lancaster,H.O. (1957) "Some properties of the bivariate normal distribution considered in the form of a contingency table",
Biometrika, 44, 289-292.

\bibitem{} Lancaster H.O. (1958): "The structure of bivariate distributions",
\emph{Annals of Mathematical Statistics}, 29, 719-736.

\bibitem{N} Nelsen R.(2006): \emph{An Introduction to Copulas}, Springer

\bibitem{} Renyi A. (1959): "On measures of dependence",
\emph{Acta Mathematica Academiae Scientiarum Hungaricae}, 10,
441-451.


\bibitem{} Volkonskii V.A., Rozanov Yu.A. (1959): "Some limit theorems for random functions I",
\emph{Theor. Probab. Appl.}, 4, 178-197.

\bibitem{} Volkonskii V.A., Rozanov Yu.A. (1961): "Some limit theorems for random functions II",
\emph{Theor. Probab. Appl.}, 6, 186-198.




\end{thebibliography}
\end{document}